\documentclass[12pt]{amsart}
\usepackage{a4wide,epsfig,amsthm,amsmath,amsfonts,amssymb,color,epic}
\makeatletter

\theoremstyle{plain}             % oder andere
\newtheorem{theorem}{Theorem}[section]

\newtheorem{corollary}[theorem]{Corollary}

\newtheorem{remark}[theorem]{Remark}

\newcommand{\spn}{\operatorname{span}}

\newcommand{\dof}{\operatorname{dof}}

\newcommand{\TT}{\operatorname{TT}}
\newcommand{\dd}{\operatorname{d}\!}
\begin{document}
\title[Low-rank approximation of continuous functions in mixed Sobolev spaces]
{Low-rank approximation of continuous functions in Sobolev spaces 
with dominating mixed smoothness}
\author{Michael Griebel}
\address{Michael Griebel,
Institut f\"ur Numerische Simulation,
Universit\"at Bonn, Friedrich-Hirzebruch-Allee 7, 53115 Bonn, Germany
and
Fraunhofer Institute for Algorithms and Scientific Computing (SCAI), 
Schloss Birlinghoven, 53754 Sankt Augustin, Germany
}
\email{griebel@ins.uni-bonn.de}
\author{Helmut Harbrecht}
\address{
Helmut Harbrecht,
Departement Mathematik und Informatik,
Universit\"at Basel, Spiegelgasse 1, 4051 Basel, Switzerland
}
\email{helmut.harbrecht@unibas.ch}
\author{Reinhold Schneider}
\address{Reinhold Schneider,
Institute of Mathematics,
Technical University of Berlin, Stra\ss{}e des 17.\ Juni 136, 10623 Berlin,
Germany
}
\email{schneidr@math.tu-berlin.de}
\date{}

\begin{abstract}
Let  $\Omega_i\subset\mathbb{R}^{n_i}$, $i=1,\ldots,m$, be 
given domains. In this article, we study the low-rank approximation 
with respect to $L^2(\Omega_1\times\dots\times\Omega_m)$ of 
functions from Sobolev spaces with dominating mixed smoothness.
To this end, we first estimate the rank of a bivariate approximation,
i.e., the rank of the continuous singular value decomposition. In 
comparison to the case of functions from Sobolev spaces with
isotropic smoothness, compare \cite{GH14,GH19}, we obtain
improved results due to the additional mixed smoothness. This 
convergence result is then used to study the tensor train 
decomposition as a method to construct multivariate low-rank 
approximations of functions from Sobolev spaces with dominating 
mixed smoothness. We show that this approach is able to beat 
the curse of dimension.
\end{abstract}

%==========================================================
\maketitle
%==========================================================
%===================================================
\section{Introduction}
%===================================================
Many problems in science and engineering lead to problems
which are defined on the tensor product of two domains
$\Omega_1\subset\mathbb{R}^{n_1}$ and $\Omega_2
\subset\mathbb{R}^{n_2}$. Examples arise from the
second moment analysis of partial differential domains
with stochastic input parameters \cite{H,HSS,ST},
two-scale homogenization \cite{AB,CDG,HS2}, radiosity 
models and radiative transfer \cite{WHS}, or space-time 
discretizations of parabolic problems \cite{GOE}.
All these problems are directly given on the product 
of two domains. Furthermore, many problems are
posed on higher-order product domains $\Omega_1\times
\cdots\times\Omega_m$ with $\Omega_i\subset\mathbb{R}^{n_i}$,
$i=1,\ldots,m$. Prominent examples are non-Newtonian flows. 
These can be modelled by a coupled system which consists 
of the Navier Stokes equation for the flow in a three-dimensional 
geometry described by $\Omega_1$ and of the Fokker-Planck 
equation in a configuration space $\Omega_2\times\cdots
\times\Omega_m$ consisting of $m-1$ spheres. Here, $m$ 
denotes the number of atoms in a chain-like molecule which 
constitutes the non-Newtonian behaviour of the flow, for 
details see \cite{BKS,LL,LOZ}. Another class of examples 
arises from uncertainty quantification, where one has the 
product of a physical domain $\Omega_1$ with a high-dimensional 
cube $\Omega_2\times\cdots\times\Omega_m = [-1,1]^{m-1}$ 
for the stochastic parameter, compare e.g.~\cite{GS}. A third 
example stems from multiscale homogenization. Then, each 
scale $i$ involves a corresponding physical domain $\Omega_i$ 
and we encounter a problem on the product domain $\Omega_1 
\times\cdots\times\Omega_m$ for multiscale homogenization 
with $m$ well-separated scales.

In this article, we therefore study the low-rank approximation
for problems posed on product domains. We start our investigations
first with the convergence of the bivariate approximation
\begin{equation}\label{eq:bivariateSVD}
%===================================
   f_R(\boldsymbol{x}_1,\boldsymbol{x}_2) 
   	= \sum_{r=1}^R g_r^{(1)}(\boldsymbol{x}_1) g_r^{(2)}(\boldsymbol{x}_2)
\end{equation}
with respect to $L^2(\Omega_1\times\Omega_2)$. 
While it is well known that the optimal rank $R$ can be 
determined by the truncated singular value decomposition,
the \emph{convergence} with respect to the rank $R$ 
is not so easy to determine. In \cite{GH14,GH19}, the 
question of the optimal rank $R$ has been answered 
in case of functions from isotropic Sobolev spaces.
But the technique used there yields no gain if additional 
smoothness is provided by Sobolev spaces with dominant 
mixed smoothness. In order to exploit such extra regularity, 
we construct in this article specific low-rank approximations 
with known convergence properties by means of sparse 
tensor product approximations. These are known to exploit 
dominating mixed smoothness in an optimal way. As a 
consequence, we are able to improve the results from 
\cite{GH14,GH19} considerably. Indeed, the decay of
the singular values is up to a factor two faster than in
case of functions with isotropic Sobolev smoothness.

We then consider the situation of bivariate approximation
if also mixed Sobolev smoothness is provided on each 
subdomain. This means that we study the low-rank approximation
\begin{equation}\label{eq:multivariateSVD}
%===================================
   f_R(\boldsymbol{x}_1,\ldots,\boldsymbol{x}_{\ell},\boldsymbol{x}_{\ell+1},\ldots,\boldsymbol{x}_m) 
   	= \sum_{r=1}^R g_r^{(1)}(\boldsymbol{x}_1,\ldots,\boldsymbol{x}_{\ell}) 
		g_r^{(2)}(\boldsymbol{x}_{\ell+1},\ldots,\boldsymbol{x}_m)
\end{equation}
if mixed Sobolev smoothness is provided not only 
between the two subdomains $\Omega_1\times\cdots\times\Omega_\ell$  
and $\Omega_{\ell+1}\times\cdots\times\Omega_m$, but additionally 
{\em within} each of these subdomains as well. Thus, we have 
mixed smoothness for the full $m$-variate situation. We like 
to mention that our findings generalize the results from 
\cite{T1,T2,T3} for periodic functions on the $m$-cube to 
arbitrary product domains. We allow moreover arbitrary 
product domains with possibly different smoothness indices 
on each subdomain. Our results, however, coincide 
with \cite{T1,T2,T3} in the simple setting of $\Omega_1 
= \cdots = \Omega_m = [0,1]$ and periodic functions 
from Sobolev spaces of dominating mixed smoothness.

After studying the convergence of the approximation
\eqref{eq:multivariateSVD}, we are ready to consider 
the tensor train approximation. The tensor train is a 
tensor format which can be used to efficiently approximate 
multivariate functions, compare \cite{Hack,Ose11}. As 
we will see, this format is able to essentially beat 
the curse of dimension in case of 
functions with dominating mixed derivatives. Note 
that the tensor train (TT) format is a particular architecture 
of the hierarchical Tucker (HT) format \cite{HaKu}. Such 
tensor representations are known in computational
quantum physics and quantum information 
theory as tensor networks or more precisely, 
as tree-based tensor networks \cite{AN}. All 
these names are used in the literature \cite{BNS}. 
The present investigations can be extended to the 
hierarchical Tucker format by using the same or similar 
ideas, cf.~\cite{SU}. This generalization is obvious, 
but requires an extended machinery of notations, 
definitions, etc. For the sake of simplicity of presentation,
we refrain from a detailed consideration here and 
merely focus on the tensor train format.

The remainder of this article is organized as follows. 
In Section~\ref{sect2}, we specify the requirements 
of multiscale hierarchies on each subdomain. They 
will be used to construct appropriate sparse tensor 
approximations in case of multivariate functions in 
Section~\ref{sect3}. In Section~\ref{sect4}, we compute 
bounds on the truncation error of the singular value 
decomposition \eqref{eq:bivariateSVD} in the case of 
functions $f\in L^2(\Omega_1\times\Omega_2)$. 
In Section~\ref{sect5}, we then consider bounds 
of the truncated singular value decomposition 
\eqref{eq:multivariateSVD} in the case of functions 
$f\in L^2(\Omega_1\times\cdots\times\Omega_m)$.
Then, in Section~\ref{sec:TT}, we use the results of 
the previous sections to establish bounds for the 
tensor train format in the continuous setting. Finally, 
we state concluding remarks in Section~\ref{sec:conclusion}.

Throughout this article, the notion ``essential'' 
in the context of complexity estimates means 
``up to logarithmic terms''. Moreover, to avoid the 
repeated use of generic but unspecified constants, 
we denote by $C \lesssim D$ that $C$ is bounded 
by a multiple of $D$ independently of parameters 
on which $C$ and $D$ may depend. Obviously, 
$C \gtrsim D$ is defined as $D \lesssim C$, and 
$C \sim D$ as $C \lesssim D$ and $C \gtrsim D$.

%===================================================
\section{Approximation on the subdomains}\label{sect2}
%===================================================
Let $\Omega\in\mathbb{R}^n$ be a sufficiently smooth, 
bounded domain.\footnote{An $n$-dimensional, smooth, 
compact, and orientable manifold in $\mathbb{R}^{n+1}$ 
can also be considered here and in the following.} We 
consider a nested sequence of finite dimensional subspaces
 \begin{equation}		\label{eq:multiscale}
%====================================================
  V_0 \subset V_1 \subset\cdots\subset V_j\subset\cdots\subset L^2(\Omega),
  \quad V_j = \operatorname{span}\{\Phi_j\},
\end{equation}
which consists of piecewise polynomial ansatz functions
$\Phi_j := \{\varphi_{j,k}:k\in\Delta_j\}$, where $\Delta_j$ 
denotes a suitable index set, such that $\dim V_j\sim 2^{jn}$ and
\begin{equation}	\label{eq:dense}
%=======================================
  L^2(\Omega) = \overline{\bigcup_{j\in\mathbb{N}_0}V_j}.
  	%\qquad V_0 = \bigcap_{j\in\mathbb{N}_0}V_j.
\end{equation}

Since we intend to approximate functions in these
spaces $V_j$, we assume that the approximation property
\begin{equation}	\label{eq:approx}
%========================================
 \inf_{v_j\in V_j}\|u-v_j\|_{L^2(\Omega)}
 	\lesssim h_j^s\|u\|_{H^s(\Omega)},
		\quad u\in H^s(\Omega),
\end{equation}
holds for $0\le s\le r$ uniformly in $j$. Here we set 
$h_j := 2^{-j}$, i.e., $h_j$ corresponds to the width of the 
mesh associated with the subspace $V_j$ on $\Omega$. 
The norm in $H^s(\Omega)$ is defined 
as usual, see \cite{Wloka} for example, while the 
integer $r$ refers to the {\em polynomial exactness}, 
that is the maximal order of polynomials which are 
locally contained in the space $V_j$. 

We now introduce a wavelet basis associated with 
the multiscale analysis \eqref{eq:multiscale} and 
\eqref{eq:dense} as follows: The wavelets $ \Psi_j := \{\psi_{j,k}:
k\in\nabla_j\}$, where $\nabla_j:=\Delta_j\setminus\Delta_{j-1}$, 
are the bases of the complementary spaces $W_j$ of 
$V_{j-1}$ in $V_j$, i.e.,
\[
  V_j = V_{j-1}\oplus W_j, \quad V_{j-1}\cap W_j = \{0\},
  	\quad W_j = \spn\{\Psi_j\}.
\]
Recursively we obtain
\[
  V_J = \bigoplus_{j=0}^J W_j,\quad W_0 := V_0,
\]
and thus, with
\[
  \Psi_J := \bigcup_{j=0}^J\Psi_j, \quad \Psi_0 := \Phi_0,
\]
we get a wavelet basis in $V_J$. A final requirement is that 
the infinite collection $\Psi := \bigcup_{j\ge 0}\Psi_j$ forms
a Riesz basis of $L^2(\Omega)$. Then, there exists also 
a biorthogonal, or dual, wavelet basis $\widetilde\Psi
= \bigcup_{j\ge 0}\widetilde\Psi_j = \{\widetilde\psi_{j,k}:
k\in\nabla_j,\,j\ge 0\}$ which defines a dual multiscale analysis, 
compare e.g.~\cite{DA} for further details. In particular, each 
function $f\in L^2(\Omega)$ admits the unique representation
\begin{equation}	\label{eq:expansion}
%======================================
  f = \sum_{j=0}^\infty\sum_{k\in\nabla_j} 
  	(f,\widetilde\psi_{j,k})_{L^2(\Omega)}\psi_{j,k}.
\end{equation}

With the definition of the projections
\[
  Q_j:L^2(\Omega)\to W_j,\quad Q_jf 
  	= \sum_{k\in\nabla_j} (f,\widetilde\psi_{j,k})_{L^2(\Omega)}\psi_{j,k}
\]
the atomic decomposition \eqref{eq:expansion} gives 
rise to the multilevel decomposition
\[
  f = \sum_{j=0}^\infty Q_jf.
\]
Then, for any $f\in H^s(\Omega)$, the approximation 
property \eqref{eq:approx} induces the estimate
\begin{equation}	\label{eq:ops}
%===================================
  \|Q_jf\|_{L^2(\Omega)}
 	\lesssim 2^{-js}\|f\|_{H^s(\Omega)},\quad 0\le s\le r.
\end{equation}

%===================================================
\section{Sparse tensor product spaces}\label{sect3}
%===================================================
Consider now two domains $\Omega_1\subset\mathbb{R}^{n_1}$
and $\Omega_2\subset\mathbb{R}^{n_2}$ with $n_1,n_2\in\mathbb{N}$. 
We aim at the approximation of functions in $L^2(\Omega_1
\times\Omega_2)$. To this end, we assume individually for each 
subdomain $\Omega_i$, $i=1,2$, the multiscale analyses
\[
  V_0^{(i)}\subset V_1^{(i)}\subset V_2^{(i)}\subset\cdots
  	\subset L^2(\Omega_i),\quad V_j^{(i)} = \spn\{\Phi_j^{(i)}\},\quad i=1,2,
\]
with associated complementary spaces
\[
  V_j^{(i)} = V_{j-1}^{(i)}\oplus W_j^{(i)}, \quad V_{j-1}^{(i)}\cap W_j^{(i)} = \{0\},
  	\quad W_j^{(i)} = \spn\{\Psi_j^{(i)}\}.
\]
Furthermore, let us denote the polynomial exactnesses 
of the spaces $V_j^{(1)}$ and $V_j^{(2)}$ by $r_1$ and 
$r_2$, respectively.

In this article, we employ the special sparse tensor 
product space\footnote{Here and in the following, the 
summation limits are in general no natural numbers and
must of course be rounded properly. We leave this to the 
reader to avoid cumbersome floor/ceil-notations.}
\begin{equation}\label{sgneu}
%=====================================
  \widehat{V}_J^\sigma := \bigoplus_{j_1\sigma+\frac{j_2}{\sigma}\le J} 
  	W_{j_1}^{(1)}\otimes W_{j_2}^{(2)}
  = \bigoplus_{j_1\sigma+\frac{j_2}{\sigma}= J} V_{j_1}^{(1)}\otimes W_{j_2}^{(2)}
\end{equation}
for an {\em arbitrary\/} parameter $\sigma>0$. In 
particular, the index pairs $(j_1,j_2)\in
\mathbb{N}_0\times \mathbb{N}_0$ of the
included tensor product spaces $W_{j_1}^{(1)}
\otimes W_{j_2}^{(2)}$ satisfy the relations
\[
 0\le j_1 \le \frac{1}{\sigma}J-\frac{1}{\sigma^2}j_2,
  \quad 0\le j_2\le\sigma J-\sigma^2 j_1.
\]
Reasonable choices of the parameter $\sigma$ 
could be as follows:
\begin{itemize}
\item
We may equilibrate the degrees of freedom in all tensor 
product spaces $W_{j_1}^{(1)}\otimes W_{j_2}^{(2)}$, 
that is the dimension $\dim(W_{j_1}^{(1)}\otimes W_{j_2}^{(2)})
= \dim(W_{j_1}^{(1)})\cdot\dim(W_{j_2}^{(2)})$, whose 
indices $(j_1,j_2)$ satisfy $j_1\sigma+j_2/\sigma=J$.
This choice leads to $\sigma = \sqrt{n_1/n_2}$.
\item
The sparse tensor product space $\widehat{V}_J^\sigma$ 
\eqref{sgneu} can be rewritten as
\[
  \widehat{V}_J^\sigma = \sum_{j_1\sigma+j_2/\sigma=J}
  	V_{j_1}^{(1)}\otimes V_{j_2}^{(2)}.
\]
Then, it can be seen easily that the choice $\sigma := \sqrt{r_1/r_2}$ 
equilibrates the approximation power of the contained tensor 
product spaces $V_{j_1}^{(1)}\otimes V_{j_2}^{(2)}$.
\item
Following the idea of an {\em equilibrated cost-benefit rate} 
(see \cite{BG}), we get the condition
\[
  2^{j_1(n_1+r_1)}\cdot 2^{j_2(n_2+r_2)} \overset{!}{=} 2^{J\cdot const}.
\]
Then, by choosing $const = \sqrt{(n_1+r_1)(n_2+r_2)}$,
we  arrive at $\sigma = \sqrt{\frac{n_1+r_1}{n_2+r_2}}$.
\end{itemize}

We now repeat the following results from \cite{GH13a} 
as they will be essential for our analysis of low-rank 
approximations in Sobolev spaces of dominating mixed 
smoothness.

\begin{theorem}[see \cite{GH13a}]\label{thm:complexity}
%========================================
The dimension of the sparse tensor product space
\begin{equation}\label{eq:Basisdarstellung}
%=======================================
  \widehat{V}_J^\sigma = \sum_{\sigma j_1+ j_2/\sigma\le J} 
  	W_{j_1}^{(1)}\otimes W_{j_2}^{(2)}
\end{equation}
is essentially $\mathcal{O}(2^{J\max\{n_1/\sigma,n_2\sigma\}})$. 
More precisely, it holds
\begin{equation}	\label{eq:sg_cost_0}
%===============================
  \dim\widehat{V}_J^\sigma\lesssim\begin{cases} 
  	2^{J\max\{n_1/\sigma,n_2\sigma\}},&\text{if $n_1/\sigma\not= n_2\sigma$},\\
	2^{Jn_2\sigma}J,&\text{if $n_1/\sigma= n_2\sigma$}.\end{cases}
\end{equation}
\end{theorem}

The constant in the estimate \eqref{eq:sg_cost_0} depends 
on the particular choice of $\sigma$. Note that the 
sparse tensor product spaces $\widehat{V}_J^\sigma$
contains essentially less degrees of freedom than the full tensor 
product space $V_{J/\sigma}^{(1)}\otimes V_{J\sigma}^{(2)}$, 
which possesses, up to a constant, $2^{J(n_1/\sigma+n_2\sigma)}$ 
degrees of freedom.

In order to determine the best choice of $\sigma$ later on, 
we need to know the rate of approximation in the sparse tensor 
spaces $\widehat{V}_J^\sigma$. To this end, for $s_1,s_2\ge 0$, 
we introduce the anisotropic Sobolev spaces
\[
  H_{mix}^{s_1,s_2}(\Omega_1\times\Omega_2)
  	:= H^{s_1}(\Omega_1)\otimes H^{s_2}(\Omega_2),
\]
equipped with the standard cross norm. As we consider here 
the tensor product of Hilbert spaces, the topological tensor 
product and the algebraic tensor product coincide. Obviously, 
the highest possible rate of convergence is attained in the 
space $H_{mix}^{r_1,r_2}(\Omega_1\times\Omega_2)$.
Therefore, in the following, we restrict ourselves without 
loss of generality to $s_1\le r_1$ and $s_2\le r_2$.

\begin{theorem}[see \cite{GH13a}]\label{thm:accuracy1}
%========================================
Let $0< s_1\le r_1$, $0< s_2\le r_2$ and $f\in H_{mix}^{s_1,s_2}
(\Omega_1\times\Omega_2)$. Then, the approximation 
\begin{equation}	\label{eq:solution1}
%=========================================
  \widehat{f}_J = \sum_{j_1\sigma + \frac{j_2}{\sigma}\le J}
	 \big(Q_{j_1}^{(1)}\otimes Q_{j_2}^{(2)}\big)f\in\widehat{V}_J^\sigma
\end{equation}
satisfies
\begin{equation}	\label{eq:convergence rate1}
%============================================
    \|f-\widehat{f}_J\|_{L^2(\Omega_1\times\Omega_2)}\lesssim
    \begin{cases} 2^{-J\min\{s_1/\sigma,s_2\sigma\}}
    	\|f\|_{H_{mix}^{s_1,s_2}(\Omega_1\times\Omega_2)},
    		&\text{if $s_1/\sigma \not=s_2\sigma$},\\
    2^{-Js_1/\sigma}\sqrt{J}\|f\|_{H_{mix}^{s_1,s_2}(\Omega_1\times\Omega_2)},
    	&\text{if $s_1/\sigma=s_2\sigma$}.\end{cases}
\end{equation}
\end{theorem}

The constant in estimate \eqref{eq:convergence rate1} 
depends again on the particular choice of $\sigma$.
Moreover, if $s_1<r_1$ and $s_2<r_2$, the factor $
\sqrt{J}$ for the case $s_1/\sigma = s_2\sigma$ in 
\eqref{eq:convergence rate1} can be removed by using 
more sophisticated estimates, compare \cite{GOS,PS1}.

Note that the combination of Theorems \ref{thm:complexity}
and \ref{thm:accuracy1} implies that, for any $\sigma$ which satisfies
the inequalities
\[
  \min\bigg\{\frac{r_1}{r_2},\frac{n_1}{n_2}\bigg\}
  	\le \sigma^2 \le \max\bigg\{\frac{r_1}{r_2},\frac{n_1}{n_2}\bigg\},
\]
the sparse tensor product spaces $\widehat{V}_J^\sigma$ 
offer essentially the same rate of convergence with respect of 
the degrees of freedom, compare \cite{GH13a}. In the following, 
we are looking for the low-rank approximation of functions. 
The key idea is to use the sparse tensor product space 
$\widehat{V}_J^\sigma$ as a tool to bound the rank properly. 

%===================================================
\section{Bivariate mixed Sobolev smoothness}\label{sect4}
%===================================================
We first like to estimate the rank which is required to represent
functions in $\widehat{V}_J^\sigma$. To this end, we make use of 
the fact that the sparse tensor product space is given as a direct 
sum of tensor products of single-scale spaces $V_{j_1}^{(1)}$ 
and complement spaces $W_{j_2}^{(2)}$ in accordance with 
\begin{equation}\label{eq:Basisdarstellung2}
%========================================
\widehat{V}_J^\sigma = \sum_{j_2 = 0}^{J\sigma} 
	\sum_{j_1 = 0}^{J/\sigma-j_2/\sigma^2} W_{j_1}^{(1)}\otimes W_{j_2}^{(2)}
	% = \sum_{j_2 = 0}^{J\sigma} V_{(J-j_2/\sigma)/\sigma}^{(1)}\otimes W_{j_2}^{(2)}
  	= \sum_{\sigma j_1+ j_2/\sigma = J} V_{j_1}^{(1)}\otimes W_{j_2}^{(2)},
\end{equation}
compare \eqref{eq:Basisdarstellung}. To this end, we make use
of the bases $\Phi_j^{(1)}$ and $\Psi_{j}^{(2)}$ and exploit that 
any function 
\[
\widehat{f}_J(\boldsymbol{x}_1,\boldsymbol{x}_2)
= \sum_{\sigma j_1+ j_2/\sigma= J} \sum_{k_1\in\Delta_{j_1}}\sum_{k_2\in\nabla_{j_2}}
  c_{(j_1,k_1),(j_2,k_2)} \varphi_{j_1,k_1}^{(1)}(\boldsymbol{x}_1)
  \psi_{j_2,k_2}^{(2)}(\boldsymbol{x}_2)\in\widehat{V}_J^\sigma
\]
can be rewritten as a tensor product function of finite rank 
in accordance with
\begin{equation}\label{eq:low-rank}
%========================================
\begin{aligned}
   \widehat{f}_J(\boldsymbol{x}_1,\boldsymbol{x}_2) &= \sum_{\sigma j_1+ j_2/\sigma= J} 
   \sum_{k_2\in\nabla_{j_2}} g_{j_1,j_2,k_2}^{(1)}(\boldsymbol{x}_1) \psi_{j_2,k_2}^{(2)}(\boldsymbol{x}_2)\\
   	&= \sum_{\sigma j_1+ j_2/\sigma= J} \sum_{k_1\in\Delta_{j_1}}
	\varphi_{j_1,k_1}^{(1)}(\boldsymbol{x}_1) g_{j_1,j_2,k_1}^{(2)}(\boldsymbol{x}_2),
\end{aligned}
\end{equation}
where 
\begin{align*}
g_{j_1,j_2,k_2}^{(1)} &= \sum_{k_1\in\Delta_{j_1}}c_{(j_1,k_1),(j_2,k_2)}\varphi_{j_1,k_1}\in V_{j_1}^{(1)},\\
g_{j_1,j_2,k_1}^{(2)} &= \sum_{k_2\in\nabla_{j_2}}c_{(j_1,k_1),(j_2,k_2)}\psi_{j_2,k_2}\in W_{j_2}^{(2)}.
\end{align*}

By switching between both representations in
\eqref{eq:low-rank} for particular combinations of
$(j_1,j_2)$, we are able to derive the following 
result.

\begin{theorem}	\label{thm:rank}
%========================================
Any function in the sparse tensor product space 
$\widehat{V}_J^\sigma$ can be represented 
as a tensor product function 
\[
   \widehat{f}_J(\boldsymbol{x}_1,\boldsymbol{x}_2) 
   	= \sum_{r=0}^{R} g_r^{(1)}(\boldsymbol{x}_1) 
		g_r^{(2)}(\boldsymbol{x}_2) 
\]
of rank $R$, where the rank is bounded by
\[
  R\lesssim R_J^\sigma:= 2^{\frac{J n_1n_2}{n_1/\sigma +\sigma n_2}}.
\]
\end{theorem}

\begin{proof}
We start with the identity \eqref{eq:Basisdarstellung2}.
Since the rank of functions in the space $V_{j_1}^{(1)}\otimes 
W_{j_2}^{(2)}$ is bounded by 
\[
\min\big\{\dim V_{j_1}^{(1)}, \dim W_{j_2}^{(2)}\big\}\sim\min\{2^{j_1n_1},2^{j_2n_2}\}, 
\]
it holds
\[
  R\lesssim\sum_{j_1\sigma+j_2/\sigma=J} \min\{2^{j_1n_1},2^{j_2n_2}\}.
\]
Assume that the equilibrium in the bracket is obtained 
for $j_1^\star n_1 = j_2^\star n_2$, then we can split the
index set $\{(j_1,j_2): j_1\sigma+j_2/\sigma=J\}$ into the 
two index sets $\mathcal{I}_1:=\{(j_1,j_2): 0\le j_1\le j_1^\star,\ 
j_2 = J\sigma-j_1\sigma^2 \}$ and $\mathcal{I}_2:=\{(j_1,j_2): 
0\le j_2\le j_2^\star,\ j_1 = J/\sigma-j_2/\sigma^2\}$. In view
of $\min\{2^{j_1n_1},2^{j_2n_2}\} = 2^{j_1n_1}$ for all $(j_1,j_2)
\in\mathcal{I}_1$ and $\min\{2^{j_1n_1},2^{j_2n_2}\} = 2^{j_2n_2}$ 
for all $(j_1,j_2)\in\mathcal{I}_2$, we obtain
\begin{equation}\label{eq:intermediateR}
%=======================================
R\lesssim\sum_{j_1=0}^{j_1^\star} 2^{j_1n_1}
	+ \sum_{j_2=0}^{j_2^\star} 2^{j_2n_2}.
\end{equation}
Since it always holds $j_1\sigma+j_2/\sigma=J$, we
can determine $j_1^\star$ from the equation
\[
  j_1^\star n_1 = (J\sigma-j_1^\star \sigma^2) n_2,
\]
i.e., 
\[
  j_1^\star = \frac{Jn_2}{n_1/\sigma +\sigma n_2}\quad\text{and}\quad
  j_2^\star = \frac{Jn_1}{n_1/\sigma +\sigma n_2}.
\]
Inserting this into \eqref{eq:intermediateR} implies the assertion
\[
  R \lesssim 2^\frac{Jn_1n_2}{n_1/\sigma +\sigma n_2}.
\]
\end{proof}

Now, by combining Theorem~\ref{thm:accuracy1} with
Theorem~\ref{thm:rank}, we can express the convergence
rate in terms of the rank $R$. This gives us an upper
bound of the truncation error of the singular value 
decomposition for functions in spaces with 
dominating mixed smoothness.

\begin{corollary}\label{coro:cost complexity}
%===================================================
Let $0< s_1\le r_1$, $0< s_2\le r_2$ and $f\in H_{mix}^{s_1,s_2}
(\Omega_1\times\Omega_2)$. Choose $\sigma> 0$ arbitrarily and set
\begin{equation}\label{eq:alpha}
%===================================================
  \beta := \frac{n_1/\sigma +\sigma n_2}{n_1n_2}
  	\min\{s_1/\sigma,s_2\sigma\} = \frac{1}{n_1 n_2} \min\bigg\{\frac{s_1n_1}{\sigma^2}+s_1n_2,
		s_2n_1+s_2n_2 \sigma^2 \bigg\},
\end{equation}
Then, there exists a rank-$R$ approximation 
\[
   f_R(\boldsymbol{x}_1,\boldsymbol{x}_2) 
   	= \sum_{r=1}^R g_r^{(1)}(\boldsymbol{x}_1) g_r^{(2)}(\boldsymbol{x}_2) 
		\in L^2(\Omega_1\times\Omega_2)
\]
which approximates $f$ as
\[
  \|f-f_R\|_{L^2(\Omega_1\times\Omega_2)}\lesssim\begin{cases}
		R^{-\beta}\|f\|_{H_{mix}^{s_1,s_2}(\Omega_1\times\Omega_2)},
		&\text{if $s_1/\sigma\not= s_2\sigma$,}\\
	R^{-\beta}\sqrt{\log R}\|f\|_{H_{mix}^{s_1,s_2}(\Omega_1\times\Omega_2)},
		&\text{if $s_1/\sigma= s_2\sigma$.}\end{cases}
\]
The constants in these estimates do depend on $s_1$, $s_2$, $n_1$, 
$n_2$, and $\sigma$, but not on the rank $R$.
\end{corollary}

\begin{proof}
Let $s_1/\sigma\not= s_2\sigma$ and observe that $R\lesssim 
R_J^\sigma = 2^\frac{Jn_1n_2}{n_1/\sigma +\sigma n_2}$ due 
to Theorem~\ref{thm:rank}. In order to estimate the convergence
with respect to the rank, we assume $R= R_J^\sigma$, which implies
\[
  R^{-\beta} = R^{-\frac{n_1/\sigma +\sigma n_2}{n_1n_2}
  	\min\{s_1/\sigma,s_2\sigma\}}\sim 2^{-J\min\{s_1/\sigma,s_2\sigma\}}.
\]
This yields the first error estimate in view of \eqref{eq:convergence rate1}.

In case of $s_1/\sigma= s_2\sigma$, the additional factor
$\sqrt{J}\lesssim\sqrt{\log R}$ needs to be inserted as a multiplicative 
factor. This completes the proof.
\end{proof}

The optimal rate of convergence with respect to
the rank is given if the expressions in the minimum 
in \eqref{eq:alpha} are balanced, i.e., if
\[
  \frac{s_1n_1}{\sigma^2}+s_1n_2 = s_2n_1+s_2n_2 \sigma^2.
\]
Straighforward calculation yields 
\[
  \sigma = \sqrt{\frac{s_1}{s_2}},
\]
which means we should equilibrate the approximation
power in the underlying sparse tensor product space. 
This would yield the rank estimate
\begin{equation}\label{eq:rate1}
%==================================
  \|f-f_R\|_{L^2(\Omega_1\times\Omega_2)}
  	\lesssim R^{-\beta}\sqrt{\log{R}}\|f\|_{H_{mix}^{s_1,s_2}(\Omega_1\times\Omega_2)}
\end{equation}
with 
\begin{equation}\label{eq:rate2}
%==================================
  \beta = \frac{s_1 n_2+s_2 n_1}{n_1n_2} = \frac{s_1}{n_1}+\frac{s_2}{n_2}.
\end{equation}

In contrast, in \cite{GH19}, we were able to prove only the rate
\[
  \beta = \max\bigg\{\frac{s_1}{n_1},\frac{s_2}{n_2}\bigg\}.
\]
Especially, the latter rate is already achieved for functions in 
the related isotropic Sobolev space
\[
  H_{iso}^{s_1,s_2}(\Omega_1\times\Omega_2)
  	:= H_{mix}^{s_1,0}(\Omega_1\times\Omega_2)
		\cap	H_{mix}^{0,s_2}(\Omega_1\times\Omega_2).
\]
As a consequence, the singular values of a function 
from $H_{mix}^{s_1,s_2}(\Omega_1\times\Omega_2)$
converge by a factor up to two faster than the singular values 
of a function from $H_{iso}^{s_1,s_2}(\Omega_1\times\Omega_2)$.

To conclude our study on the bivariate situation, we state the 
following result for the fully discrete singular value decomposition 
(where the eigenfunctions are approximated with respect to
single-scale spaces of fixed level of resolution as introduced 
in Section \ref{sect2}):

\begin{corollary}	\label{thm:work-svd}
%===================================
The number of degrees of freedom, which is needed to approximate 
a function $f \in H_{mix}^{s_1,s_2}(\Omega_1\times\Omega_2)$ by 
the truncated singular value decomposition to a prescribed accuracy 
$\varepsilon$, is
\begin{equation}	\label{eq:work(svd)}
%==========================================
  \dof_{svd}(\varepsilon)\sim\varepsilon^{-\frac{n_1n_2}{s_1 n_2+s_2 n_1}}
  	\varepsilon^{-\max\{\frac{n_1}{\min\{s_1,r_1\}},\frac{n_2}{\min\{s_2,r_2\}}\}}.	
\end{equation}
\end{corollary}

In contrast, the sparse tensor product approximation 
requires essentially
\begin{equation}	\label{eq:work(sg)aniso}
%=========================================
  \dof_{sg}(\varepsilon)\sim\varepsilon^{-\max\left\{\frac{n_1}{\min\{s_1,r_1\}},
  	\frac{n_2}{\min\{s_2,r_2\}}\right\}}
\end{equation}
degrees of freedom to approximate functions from the anisotropic
Sobolev space $H_{mix}^{s_1,s_2}(\Omega_1\times\Omega_2)$,
compare \cite{GH13a}. We emphasize, however, that the multiscale 
representation \eqref{eq:low-rank} of the singular value decomposition 
has the same complexity as the sparse tensor product approximation
\eqref{eq:work(sg)aniso}. This implies that the eigenfunctions 
can be stored in a compressed format to further improve the 
complexity \eqref{eq:work(svd)} towards that of the 
respective sparse tensor product approximation.

\begin{remark}
We should comment on the sharpness of our findings. At least 
in the simple bivariate situation with $\Omega := \Omega_1 = 
\Omega_2\subset\mathbb{R}^n$, the rank estimate given by 
\eqref{eq:rate1}, \eqref{eq:rate2} is sharp. This is seen by 
considering Mat\`ern kernels $k_\nu(\boldsymbol{x}_1,
\boldsymbol{x}_2)$, also called Sobolev splines, where $\nu\ge1/2$ 
is the smoothness parameter, compare \cite{FY11,MAT}. They are 
known to be the reproducing kernels of the Sobolev spaces $H^{\nu+n/2}
(\Omega)$, hence they are themselves elements of $H^{2\nu+n}(\Omega
\times\Omega)$ and consequently also elements of $H_{mix}^{s_1,s_2}
(\Omega\times\Omega)$ for any $s_1+s_2 = 2\nu+n$. The rank estimate 
given by \eqref{eq:rate1}, \eqref{eq:rate2} coincides then essentially with 
Weyl's law \cite{Weyl}.
\end{remark}

%===================================================
\section{Multivariate mixed Sobolev smoothness}\label{sect5}
%===================================================
We shall generalize the results from the previous section to the
bivariate approximation of functions which provide multivariate 
dominating mixed derivatives. This means that we consider $m$ 
domains $\Omega_i\subset\mathbb{R}^{n_i}$ with $n_i\in\mathbb{N}$ 
for all $i=1,2,\ldots,m$ and aim at the bivariate approximation of 
functions in the anisotropic Sobolev spaces
\[
  {\bf H}^{\bf s}(\boldsymbol\Omega)
  	:= H^{s_1}(\Omega_1)\otimes H^{s_2}(\Omega_2)\otimes\cdots
		\otimes H^{s_m}(\Omega_m),
\]
which are defined on the $m$-fold product domain 
$\boldsymbol\Omega:=\Omega_1\times\Omega_2
\times\cdots\times\Omega_m$. To this end, we 
introduce the generalized $m$-fold sparse 
tensor product space
\begin{equation*}
%=====================================
  \widehat{V}_J^{\boldsymbol\alpha} 
  	:= \bigoplus_{\boldsymbol\alpha^T{\bf j}\le J} 
  W_{j_1}^{(1)}\otimes W_{j_2}^{(2)}\otimes\cdots\otimes W_{j_m}^{(m)}
\end{equation*}
for an {\em arbitrary\/} vector $\boldsymbol\alpha=
(\alpha_1,\alpha_2,\ldots,\alpha_m)>{\bf 0}$ and 
${\bf j}=(j_1,j_2,\ldots,j_m)\in\mathbb{N}_0^m$.

In accordance with \cite{GH13b}, we have the following 
approximation results with respect to the sparse tensor
product space $\widehat{V}_J^{\boldsymbol\alpha}$.

\begin{theorem}[see \cite{GH13b}]\label{thm:accuracy2}
%========================================
Let ${\bf 0}\le {\bf s}\le {\bf r}$ and $f\in {\bf H}^{\bf s}
(\boldsymbol\Omega)$. Then, the projector 
\begin{equation}	\label{eq:solution2}
%=========================================
  \widehat{Q}_J^{\boldsymbol\alpha}:{\bf H}^{\bf s}(\boldsymbol\Omega)
  	\to\widehat{\bf V}_J^{\boldsymbol\alpha},\quad \widehat{Q}_J^{\boldsymbol\alpha}f
		= \sum_{\boldsymbol\alpha^T{\bf j}\le J}\big(Q_{j_1}^{(1)}\otimes 
			Q_{j_2}^{(2)}\otimes\cdots\otimes Q_{j_m}^{(m)}\big)f
\end{equation}
on the sparse tensor product space 
$\widehat{\bf V}_J^{\boldsymbol\alpha}$ satisfies
\begin{equation}	\label{eq:convergence rate2}
%==========================================================
  \|(I-\widehat{Q}_J^{\boldsymbol\alpha})f\|_{L^2(\boldsymbol\Omega)}\lesssim
    2^{-J\min\{\frac{s_1}{\alpha_1},\frac{s_2}{\alpha_2},\ldots,\frac{s_m}{\alpha_m}\}} 
    J^{(P-1)/2}\|f\|_{{\bf H}^{\bf s}(\boldsymbol\Omega)}.
\end{equation}
Here, $P$ counts how often the minimum is attained 
in the exponent.
\end{theorem}

After having identified the approximation power of the $m$-fold sparse 
tensor product space $\widehat{V}_J^{\boldsymbol\alpha}$ when 
representing a given function $f\in {\bf H}^{\bf s}(\boldsymbol\Omega)$,
we shall next estimate the rank of this sparse tensor product 
approximation like in the bivariate situation. Therefore, we
again make use of the bases $\Phi_j^{(1)}$ and $\Psi_j^{(i)}$, 
$i=2,3,\ldots,m$, and consider a given function
\[
 \widehat{f}_J(\boldsymbol{x}_1,\ldots,\boldsymbol{x}_m) =
 \sum_{\boldsymbol\alpha^T{\bf j} = J}\sum_{k_1\in\Delta_{j_1}}
  \sum_{k_2\in\nabla_{j_2}}\!\!\!\dots\!\!\!\sum_{k_m\in\nabla_{j_m}}
 c_{{\bf j},{\bf k}}\varphi_{j_1,k_1}^{(1)}(\boldsymbol{x}_1)\psi_{j_2,k_2}^{(2)}(\boldsymbol{x}_2)
  	\cdots\psi_{j_m,k_m}^{(2)}(\boldsymbol{x}_m)\in\widehat{V}_J^{\boldsymbol\alpha}.
\]
We intend to estimate the rank when separating the variables $(\boldsymbol{x}_1,
\ldots,\boldsymbol{x}_\ell)$ and $(\boldsymbol{x}_{\ell+1},\ldots,\boldsymbol{x}_m)$,
where $1\le\ell< m$ can be chosen arbitrary, but is fixed throughout this section. 
To this end, we have to bound the rank in each of the tensor product spaces 
$V_{j_1}^{(1)}\otimes W_{j_2}^{(2)}\otimes\dots\otimes W_{j_m}^{(m)}$ 
when splitting the first $\ell$ variables from the last $m-\ell$ variables.
The rank is given by the minimum of the dimensions of the particular 
subspaces, i.e.,  
\[
  \operatorname{rank}(V_{j_1}^{(1)}\otimes W_{j_2}^{(2)}\otimes\dots\otimes W_{j_m}^{(m)})
  	= \min\big\{2^{j_1 n_1+\dots+j_\ell n_\ell},2^{j_{\ell+1} n_{\ell+1}+\dots+j_m n_m}\big\}.
\]
We now have to sum up over all admissible $j_1,j_2,\ldots,j_m$, 
i.e., we have to compute the sum
\[
  \sum_{\boldsymbol\alpha^T{\bf j} = J}\min\big\{2^{j_1 n_1+\dots+j_\ell n_\ell},
  	2^{j_{\ell+1} n_{\ell+1}+\dots+j_m n_m}\big\}.
\]
In order to bound this sum, we abbreviate 
\begin{equation}\label{eq:max}
%===================================
  \underline{m} := \max_{i\le\ell}\bigg\{\frac{n_i}{\alpha_i}\bigg\},
  \quad \overline{m} := \max_{i>\ell}\bigg\{\frac{n_i}{\alpha_i}\bigg\}
\end{equation}
and note that for given level $\tilde{J}$ with $0\le\tilde{J}\le J$ 
there holds
\begin{gather*}
\max_{\sum_{i\le\ell} \alpha_i j_i=\tilde{J}}
j_1 n_1+\dots+j_\ell n_\ell \le \tilde{J}\underline{m},\\
\max_{\sum_{i>\ell}\alpha_i j_i= \tilde{J}} 
j_{\ell+1} n_{\ell+1}+\dots+j_m n_m
\le\tilde{J}\overline{m}.
\end{gather*}
We conclude
\[
  \sum_{\begin{smallmatrix}\sum_{i\le\ell} \alpha_i j_i=\tilde{J}\\
	\sum_{i>\ell}\alpha_i j_i= J-\tilde{J}\end{smallmatrix}} 
		\min\big\{2^{j_1 n_1+\dots+j_\ell n_\ell},2^{j_{\ell+1} n_{\ell+1}+\dots+j_m n_m}\big\}
  \le J^{\max\{P_1,P_2\}-1} \min\big\{ 2^{\tilde{J}\underline{m}},2^{(J-\tilde{J})\overline{m}}\big\},
\]
where $P_1\le\ell$ and $P_2\le m-\ell$ count how often the maxima
$\underline{m}$ and $\overline{m}$, respectively, are attained 
in \eqref{eq:max}. Thus, we derive
\begin{align*}
  &\sum_{\boldsymbol\alpha^T{\bf j} = J}\min\big\{2^{j_1 n_1+\dots+j_\ell n_\ell},
  	2^{j_{\ell+1} n_{\ell+1}+\dots+j_m n_m}\big\}\\
  &\qquad=\sum_{\tilde{J}=0}^J \sum_{\begin{smallmatrix}\sum_{i\le\ell} \alpha_i j_i=\tilde{J}\\
	\sum_{i>\ell}\alpha_i j_i= J-\tilde{J}\end{smallmatrix}} 
		\min\big\{2^{j_1 n_1+\dots+j_\ell n_\ell},2^{j_{\ell+1} n_{\ell+1}+\dots+j_m n_m}\big\}\\
  &\qquad\le J^{\max\{P_1,P_2\}-1} \sum_{\tilde{J}=0}^J
  	\min\big\{2^{\tilde{J}\underline{m}},2^{(J-\tilde{J})\overline{m}}\big\}.
\end{align*}
The minimum switches for the level $J\overline{m}/(\underline{m}+\overline{m})$,
which leads to
\begin{align*}
  &\sum_{\boldsymbol\alpha^T{\bf j} = J}\min\big\{2^{j_1 n_1+\dots+j_\ell n_\ell},
  	2^{j_{\ell+1} n_{\ell+1}+\dots+j_m n_m}\big\}\\
  &\qquad\le J^{\max\{P_1,P_2\}-1}\Bigg[\sum_{\tilde{J}=0}^{J\overline{m}/(\underline{m}+\overline{m})} 
	\min\big\{2^{\tilde{J}\underline{m}},2^{(J-\tilde{J})\overline{m}}\big\}
  +\sum_{\tilde{J}=J\overline{m}/(\underline{m}+\overline{m})}^J
	\min\big\{2^{\tilde{J}\underline{m}},2^{(J-\tilde{J})\overline{m}}\big\}\Bigg]\\
  &\qquad\le J^{\max\{P_1,P_2\}-1}\Bigg[\sum_{\tilde{J}=0}^{J\overline{m}/(\underline{m}+\overline{m})} 
	2^{\tilde{J}\underline{m}} + \sum_{\tilde{J}=0}^{J\underline{m}/(\underline{m}+\overline{m})} 
	2^{\tilde{J}\overline{m}}\Bigg]\\
	&\qquad\lesssim J^{\max\{P_1,P_2\}-1} 2^{J\underline{m}\overline{m}/(\underline{m}+\overline{m})}.
\end{align*}
This determines the rank of the sparse tensor product approximation,
which we exploit in the following theorem.

\begin{theorem}	\label{thm:cost complexity L}
%===================================================
Let $0\le{\bf s}\le{\bf r}$ and $f\in {\bf H}^{\bf s}(\boldsymbol\Omega)$. 
Then, there exists a rank-$R$ approximation 
\[
   f_R(\boldsymbol{x}_1,\ldots,\boldsymbol{x}_m) = \sum_{r=0}^R g_r^{(1)}(\boldsymbol{x}_1,\ldots,\boldsymbol{x}_\ell) 
   	g_r^{(2)}(\boldsymbol{x}_{\ell+1},\ldots,\boldsymbol{x}_m) \in L^2(\boldsymbol\Omega)
\]
which approximates $f$ as
\begin{equation}\label{eq:approx_esti}
%========================================
  \|f-f_R\|_{L^2(\boldsymbol\Omega)}\lesssim
		R^{-\beta} (\log R_J)^{(m-1)/2+\beta\max\{P_1,P_2\}-1} 
			\|f\|_{{\bf H}^{\bf s}(\boldsymbol\Omega)}.
\end{equation}
Here, the rate $\beta$ is given by
\begin{equation}\label{eq:approx_rate}
%========================================
\beta = \min_{i\le\ell}\bigg\{\frac{s_i}{n_i}\bigg\}
  	+ \min_{i>\ell}\bigg\{\frac{s_i}{n_i}\bigg\}
\end{equation}
and $P_1$ and $P_2$ count how often the first and second 
minimum is attained. Moreover, the constants in this error estimate 
do depend on $s_i$ and $n_i$, $i=1,2,\ldots,m$, but not on the rank $R$.
\end{theorem}

\begin{proof}
As in the bivariate situation, we shall equilibrate the approximation 
power of the extremal spaces in the sparse tensor product construction,
which means that we choose $\alpha_i = s_i$. Thus, the rank $R_J$ of
  \[
 \widehat{Q}_J^{\boldsymbol s} f(\boldsymbol{x}_1,\ldots,\boldsymbol{x}_m) =
 \sum_{{\bf s}^T{\bf j} = J}\sum_{k_1\in\Delta_{j_1}}
  \sum_{k_2\in\nabla_{j_2}}\!\!\!\dots\!\!\!\sum_{k_m\in\nabla_{j_m}}
  c_{{\bf j},{\bf k}}\varphi_{j_1,k_1}^{(1)}(\boldsymbol{x}_1)\psi_{j_2,k_2}^{(2)}({\bf x_2})
  	\cdots\psi_{j_m,k_m}^{(2)}(\boldsymbol{x}_m)\in\widehat{V}_J^{\bf s}
\]
is essentially bounded by
\[
  R_J\lesssim 2^{J\underline{m}\overline{m}/(\underline{m}+\overline{m})}J^{\max\{P_1,P_2\}-1},
  \quad \underline{m} := \max_{i\le\ell}\bigg\{\frac{n_i}{s_i}\bigg\},
  \quad \overline{m} := \max_{i>\ell}\bigg\{\frac{n_i}{s_i}\bigg\}.
\]
This is equivalent to $R_J\lesssim 2^{J/\beta}J^{\max\{P_1,P_2\}-1}$ 
with $\beta$ as defined in \eqref{eq:approx_rate}. Hence, we have
\begin{equation}\label{eq:rank}
%======================================
  2^{-J}\lesssim \bigg(\frac{R_J}{J^{\max\{P_1,P_2\}-1}}\bigg)^{-\beta} 
  	= R_J^{-\beta}J^{\beta(\max\{P_1,P_2\}-1)}.
\end{equation}
As the convergence rate in \eqref{eq:convergence rate2} becomes
essentially $2^{-J}\log{J^{(m-1)/2}}$ due to the specific choice of 
$\boldsymbol\alpha = \boldsymbol{s}$, we can insert \eqref{eq:rank} 
to get the convergence rate $R_J^{-\beta} J^{(m-1)/2+\beta\max\{P_1,P_2\}-1}$. 
By using finally $J\lesssim \log R_J$, we conclude the desired 
error estimate \eqref{eq:approx_esti}.
\end{proof}

\begin{remark}
(i) As we have seen, in case of multivariate mixed Sobolev 
smoothness, the decay of the singular values is essentially 
the same independent of the dimension of the two domains 
$\Omega_1\times\cdots\times\Omega_{\ell}$ and 
$\Omega_{\ell+1}\times\cdots\times\Omega_m$ with 
which the bivariate approximation is performed. Nonetheless,
for the approximation of the left and right eigenfunctions, this 
holds only true if they are represented in the respective sparse 
tensor product spaces instead of the full tensor product spaces.\\
(ii) In order to optimally estimate the decay of the eigenvalues, 
one has to choose ansatz spaces which provide sufficient
polynomial exactness, i.e., given $f\in {\bf H}^{\bf s}(\boldsymbol
\Omega)$, one chooses ansatz spaces such that it holds 
${\bf r}\ge{\bf s}$ in Theorem~\ref{thm:cost complexity L}.
\end{remark}

%===================================================
\section{Tensor train format}\label{sec:TT}
%===================================================
We shall next apply the previous results for estimating the 
cost of the tensor train approximation in the continuous case. 
We basically proceed as in \cite{GH21} and successively apply 
the singular value decomposition as studied in the previous 
section. This corresponds to the algorithmic
procedure in practical computations. Related results, obtained
by other proof techniques, can be found in \cite{AN,BNS,Hack,SU}.

%===================================================
\subsection{Tensor train decomposition}
%===================================================
Given a function $f\in {\bf H}^{\bf s}(\boldsymbol\Omega)$, we
separate in the first step of the tensor train decomposition the variables 
$\boldsymbol{x}_1\in\Omega_1$ and $(\boldsymbol{x}_2,\ldots,\boldsymbol{x}_m)
\in\Omega_2\times\dots\times\Omega_m$ by the singular value decomposition
\[
  f(\boldsymbol{x}_1,\boldsymbol{x}_2,\ldots,\boldsymbol{x}_n)
  	= \sum_{\alpha_1=1}^\infty\sqrt{\lambda_1(\alpha_1)}\varphi_1(\boldsymbol{x}_1,\alpha_1)
		\psi_1(\alpha_1,\boldsymbol{x}_2,\ldots,\boldsymbol{x}_m).
\]
Herein, $\{\varphi_1(\alpha_1)\}_{\alpha_1\in\mathbb{N}}$ is an orthonormal
basis of $L^2(\Omega_1)$, $\{\psi_1(\alpha_1)\}_{\alpha_1\in\mathbb{N}}$ is 
an orthonormal basis of $L^2(\Omega_2\times\cdots\times\Omega_m)$,
and $\big\{\sqrt{\lambda_1(\alpha_1)}\big\}_{\alpha_1\in\mathbb{N}}$ is a square 
summable sequence of singular values.

Next, since
\[
   \bigg[\sqrt{\lambda_1(\alpha_1)}\psi_1(\alpha_1)\bigg]_{\alpha_1=1}^\infty
   \in \ell^2(\mathbb{N})\otimes L^2(\Omega_2\times\dots\times\Omega_m),
\]
we can separate in the second step of the tensor train decomposition 
$(\alpha_1,\boldsymbol{x}_2)\in\mathbb{N}\times\Omega_2$ from 
$(\boldsymbol{x}_3,\ldots,\boldsymbol{x}_m)\in\Omega_3\times
\dots\times\Omega_m$ by means of a second singular value 
decomposition. This leads to
\begin{equation}\label{eq:tt-step}
%===================================
\begin{aligned}
 &\bigg[\sqrt{\lambda_1(\alpha_1)}\psi_1(\alpha_1,\boldsymbol{x}_2,
 		\ldots,\boldsymbol{x}_m)\bigg]_{\alpha_1=1}^\infty\\
  	&\qquad\qquad= \sum_{\alpha_2=1}^\infty\sqrt{\lambda_2(\alpha_2)}
		\bigg[\varphi_2(\alpha_1,\boldsymbol{x}_2,\alpha_2)\bigg]_{\alpha_1=1}^\infty
			\psi_2(\alpha_2,\boldsymbol{x}_3,\ldots,\boldsymbol{x}_m).
\end{aligned}
\end{equation}
Herein, $\big\{[\varphi_2(\alpha_1,\alpha_2)]_{\alpha_1\in\mathbb{N}}
\big\}_{\alpha_2\in\mathbb{N}}$ is an orthonormal basis of $\ell^2(\mathbb{N})\otimes 
L^2(\Omega_2)$, $\{\psi_2(\alpha_2)\}_{\alpha_1\in\mathbb{N}}$ is an orthonormal basis 
of $L^2(\Omega_3\times\cdots\times\Omega_m)$, and 
$\big\{\sqrt{\lambda_2(\alpha_2)}\big\}_{\alpha_2\in\mathbb{N}}$ is 
a square summable sequence of singular values.

By repeating the second step and successively separating
$(\alpha_{j-1},\boldsymbol{x}_j)\in\mathbb{N}\times\Omega_j$ from 
$(\boldsymbol{x}_{j+1},\ldots,\boldsymbol{x}_m)\in\Omega_{j+1}
\times\dots\times\Omega_m$ for $j=3,\ldots,m-1$, we finally 
arrive at the representation
\begin{align*}
  f(\boldsymbol{x}_1,\ldots,\boldsymbol{x}_m) &= \sum_{\alpha_1=1}^\infty\cdots\sum_{\alpha_{m-1}=1}^\infty
  	\varphi_1(\alpha_1,\boldsymbol{x}_1)\varphi_2(\alpha_1,\boldsymbol{x}_2,\alpha_2)\\
  &\hspace*{2.5cm}\cdots\varphi_{m-1}(\alpha_{m-2},\boldsymbol{x}_{m-1},\alpha_{m-1})
	\varphi_m(\alpha_{m-1},\boldsymbol{x}_m),
\end{align*}
where 
\[
  \varphi_m(\alpha_{m-1},\boldsymbol{x}_m) 
  	= \sqrt{\lambda_{m-1}(\alpha_{m-1})}\psi_{m-1}(\alpha_{m-1},\boldsymbol{x}_m).
\]

In practice, we truncate the singular value decomposition in step $j$ 
after $r_j$ terms, thus arriving at the finite dimensional representation
\begin{align*}
  f_{r_1,\ldots,r_{m-1}}^{\TT}(\boldsymbol{x}_1,\ldots,\boldsymbol{x}_m)
  	&= \sum_{\alpha_1=1}^{r_1}\cdots\sum_{\alpha_{m-1}=1}^{r_{m-1}}\varphi_1(\alpha_1,\boldsymbol{x}_1)
		\varphi_2(\alpha_1,\boldsymbol{x}_2,\alpha_2)\\&\hspace*{2cm}\cdots
		\varphi_{m-1}(\alpha_{m-2},\boldsymbol{x}_{m-1},\alpha_{m-1})
			\varphi_m(\alpha_{m-1},\boldsymbol{x}_m).
\end{align*}
One readily infers by using Pythagoras' theorem that
the truncation error is then given by
\[
  \|f-f_{r_1,\ldots,r_{m-1}}^{\TT}\|_{L^2(\Omega_1\times\dots\times\Omega_m)}
  	\le\sqrt{\sum_{j=1}^{m-1}\sum_{\alpha_j=r_j+1}^\infty \lambda_j(\alpha_j)},
\]
compare \cite[Proposition 9]{M16}. Note that, for 
$j\ge 2$, the singular values $\{\lambda_j(\alpha)\}_{\alpha\in\mathbb{N}}$ 
in this estimate do not coincide with the singular values from the 
untruncated tensor train decomposition due to the truncation
after the ranks $r_j$.

%===================================================
\subsection{Regularity}
%===================================================
We shall next give bounds on the input data for the singular 
value decomposition and the truncation rank in each step 
of the tensor train decomposition. We mention that our 
analysis covers the computational practice: We compute 
successively the truncated singular value decomposition 
with prescribed accuracy $\varepsilon>0$ as it is done 
in practice.

In the first step of the tensor train decomposition, 
we compute the function 
\[
  \boldsymbol{g}_1(\boldsymbol{x}_2,\ldots,\boldsymbol{x}_m) 
  	:= \bigg[\sqrt{\lambda_1(\alpha_1)}
		\psi_1(\alpha_1,\boldsymbol{x}_2,
			\ldots,\boldsymbol{x}_m)\bigg]_{\alpha_1=1}^{r_1}.
\]
It satisfies
\begin{align*}
  \|\boldsymbol{g}_1\|_{[{\bf H}^{\boldsymbol{t}_1}(\boldsymbol\Upsilon_1)]^{r_1}}^2
  &= \sum_{\alpha_1=1}^{r_1}\lambda_1(\alpha_1)
		\|\psi_1(\alpha_1)\|_{{\bf H}^{\boldsymbol{t}_1}(\boldsymbol\Upsilon_1)}^2\\
  &= \sum_{\alpha_1=1}^{r_1}\lambda_1(\alpha_1)
		\|\varphi_1(\alpha_1)\otimes\psi_1(\alpha_1)\|_{L^2(\Omega_1)\otimes 
			{\bf H}^{\boldsymbol{t}_1}(\boldsymbol\Upsilon_1)}^2\\
  &= \Bigg\|\sum_{\alpha_1=1}^{r_1}\sqrt{\lambda_1(\alpha_1)}
		\varphi_1(\alpha_1)\otimes\psi_1(\alpha_1)\Bigg\|_{L^2(\Omega_1)\otimes 
			{\bf H}^{\boldsymbol{t}_1}(\boldsymbol\Upsilon_1)}^2,
\end{align*}
where we used the notation ${\bf H}^{\boldsymbol{t}_1}(\boldsymbol\Upsilon_1) 
:= H^{s_2}(\Omega_2)\otimes\cdots\otimes H^{s_m}(\Omega_m)$. Vice 
versa, we have by using Pythagoras' theorem 
\begin{align*}
\|f\|_{L^2(\Omega_1)\otimes {\bf H}^{\boldsymbol{t}_1}(\boldsymbol\Upsilon_1)}^2
&= \Bigg\|\sum_{\alpha_1=1}^{\infty}\sqrt{\lambda_1(\alpha_1)}
		\varphi_1(\alpha_1)\otimes\psi_1(\alpha_1)\Bigg\|_{L^2(\Omega_1)\otimes 
			{\bf H}^{\boldsymbol{t}_1}(\boldsymbol\Upsilon_1)}^2\\
&=\Bigg\|\sum_{\alpha_1=1}^{r_1}\sqrt{\lambda_1(\alpha_1)}
		\varphi_1(\alpha_1)\otimes\psi_1(\alpha_1)\Bigg\|_{L^2(\Omega_1)\otimes 
			{\bf H}^{\boldsymbol{t}_1}(\boldsymbol\Upsilon_1)}^2\\
  &\qquad+ \Bigg\|\sum_{\alpha_1=r_1+1}^{\infty}\sqrt{\lambda_1(\alpha_1)}
		\varphi_1(\alpha_1)\otimes\psi_1(\alpha_1)\Bigg\|_{L^2(\Omega_1)\otimes 
			{\bf H}^{\boldsymbol{t}_1}(\boldsymbol\Upsilon_1)}^2	.
\end{align*}
Putting both estimates together yields
\[
  \|\boldsymbol{g}_1\|_{[{\bf H}^{\boldsymbol{t}_1}(\boldsymbol\Upsilon_1)]^{r_1}}
	\le\|f\|_{{\bf H}^{\boldsymbol{s}}(\boldsymbol\Omega)}.
\]

In the $j$-th step of the tensor train decomposition, $j=2,3,\ldots,m-1$, one 
computes the singular value decomposition for the vector-valued function
\[
  \boldsymbol{g}_{j-1}(\boldsymbol{x}_j,\ldots,\boldsymbol{x}_m) 
  	:= \bigg[\sqrt{\lambda_{j-1}(\alpha_{j-1})}
		\psi_{j-1}(\alpha_{j-1},\boldsymbol{x}_j,
			\ldots,\boldsymbol{x}_m)\bigg]_{\alpha_{j-1}=1}^{r_{j-1}}.
\]
This leads to the representation
\begin{equation}\label{eq:g_{j-1}}
%=======================================
  \boldsymbol{g}_{j-1}(\boldsymbol{x}_j,\ldots,\boldsymbol{x}_m) = 
  \Bigg[\sum_{\alpha_j=1}^\infty\sqrt{\lambda_j(\alpha_j)}\varphi_{j}(\alpha_{j-1},\boldsymbol{x}_j,\alpha_j)
  \psi_j(\alpha_j,\boldsymbol{x}_{j+1},\ldots,\boldsymbol{x}_m)\Bigg]_{\alpha_{j-1}=1}^{r_{j-1}}
\end{equation}
which separates $(\alpha_{j-1},\boldsymbol{x}_j)$ from 
$(\boldsymbol{x}_{j+1},\ldots,\boldsymbol{x}_m)$, coupled only
by $\alpha_j$.

We truncate \eqref{eq:g_{j-1}} after $r_j$ terms and derive 
the new function
\[
 \boldsymbol{g}_j(\boldsymbol{x}_{j+1},\ldots,\boldsymbol{x}_m) := 
  \bigg[\sqrt{\lambda_j(\alpha_j)}\psi_j(\alpha_j,\boldsymbol{x}_{j+1},\ldots,\boldsymbol{x}_m)\bigg]_{\alpha_j=1}^{r_j}.
\]
For $\boldsymbol{t}_j := (s_{j+1},\ldots,s_m)$, we find that
\begin{align*}
\|\boldsymbol{g}_j\|_{[{\bf H}^{\boldsymbol{t}_j}(\boldsymbol\Upsilon_j)]^{r_j}}^2
&= \sum_{\alpha_j=1}^{r_j}\lambda_j(\alpha_j)\|\psi_j(\alpha_j)\|_{{\bf H}^{\boldsymbol{t}_j}(\boldsymbol\Upsilon_j)}^2\\
&= \sum_{\alpha_{j-1}=1}^{r_{j-1}} \sum_{\alpha_j=1}^{r_j}\lambda_j(\alpha_j)\|\varphi_j(\alpha_{j-1},\alpha_j)
\otimes\psi_j(\alpha_j)\|_{L^2(\Omega_j)\otimes {\bf H}^{\boldsymbol{t}_j}(\boldsymbol\Upsilon_j)}^2\\
&= \Bigg\|\sum_{\alpha_{j-1}=1}^{r_{j-1}} \sum_{\alpha_j=1}^{r_j}\sqrt{\lambda_j(\alpha_j)}\varphi_j(\alpha_{j-1},\alpha_j)
\otimes\psi_j(\alpha_j)\Bigg\|_{L^2(\Omega_j)\otimes {\bf H}^{\boldsymbol{t}_j}(\boldsymbol\Upsilon_j)}^2.\
\end{align*}
by exploiting the orthonormality of the vector-valued functions 
$[\varphi_j(\alpha_{j-1},\alpha_j)]_{\alpha_{j-1}=1}^{r_{j-1}}$, $\alpha_j
=1,\ldots,r_j$. Analogously to above, we infer that
\[
\|\boldsymbol{g}_j\|_{[{\bf H}^{\boldsymbol{t}_j}(\boldsymbol\Upsilon_j)]^{r_j}}
\le\|\boldsymbol{g}_{j-1}\|_{[L^2(\Omega_j)\otimes 
{\bf H}^{\boldsymbol{t}_j}(\boldsymbol\Upsilon_j)]^{r_{j-1}}},
\]
hence, we conclude
\begin{equation}\label{eq:bound}
%==========================================
\|\boldsymbol{g}_j\|_{[{\bf H}^{\boldsymbol{t}_j}(\boldsymbol\Upsilon_j)]^{r_j}}
\le\|f\|_{{\bf H}^{\boldsymbol{s}}(\boldsymbol\Omega)}
\quad \text{for all $j=1,2,\ldots,m-1$}.
\end{equation}

%===================================================
\subsection{Truncation ranks}
%===================================================
After having proven the regularity of the functions
${\bf g}_j$ for all $j=1,\ldots,m-1$, we shall next determine 
the truncation ranks. To this end, consider a prescribed
approximation accuracy $\varepsilon > 0$.

In the we first step of the tensor train decomposition, i.e.,
for $j=1$, we can immediately apply Theorem \ref{thm:cost complexity L} 
to get essentially the decay $r_1^{-\beta_1}$ in the truncation error, 
where the rate $\beta_1$ is given by
\begin{equation}\label{eq:beta1}
%================================
\beta_1 = \bigg\{\frac{s_1}{n_1}\bigg\}
  	+ \min_{i=2}^m\bigg\{\frac{s_i}{n_i}\bigg\}.
\end{equation}
Hence, we have to choose $r_1:=\varepsilon^{-1/\beta_1}$
to essentially get the truncation error $\varepsilon$ in the first step. 

In the $j$-th step of the tensor train decomposition, we
set $\boldsymbol{y}_j:=(\boldsymbol{x}_{j+1},\ldots,\boldsymbol{x}_m)
\in\boldsymbol\Upsilon_j:=\Omega_{j+1}\times\cdots\times\Omega_m$ and
compute the kernel function $\kappa_j\in L^2(\boldsymbol\Upsilon_j,\boldsymbol\Upsilon_j)$
given by
\[
  \kappa_j(\boldsymbol{y}_j,\boldsymbol{y}_j')
  := \sum_{\alpha_{j-1}=1}^{r_{j-1}} \lambda_{j-1}(\alpha_{j-1})\int_{\Omega_j}
  	\psi_{j-1}(\alpha_{j-1},\boldsymbol{x}_j,\boldsymbol{y}_j)
	\psi_{j-1}(\alpha_{j-1},\boldsymbol{x}_j,\boldsymbol{y}_j')\dd\boldsymbol{x}_j.
\]
This gives rise to the spectral decomposition
\begin{equation}\label{eq:spectral}
%==================================
\kappa_j(\boldsymbol{y}_j,\boldsymbol{y}_j') = \sum_{\alpha_j=1}^{\infty}
	\lambda_j(\alpha_j)\psi_{j}(\alpha_{j},\boldsymbol{y}_j)\psi_{j}(\alpha_{j},\boldsymbol{y}_j').
\end{equation}
By setting 
\[
  \varphi_{j}(\alpha_{j-1},\boldsymbol{x}_j,\alpha_j) := \frac{\sqrt{\lambda_{j-1}(\alpha_{j-1})}}{\sqrt{\lambda_j(\alpha_j)}}
  	\int_{\boldsymbol\Upsilon_j}\psi_{j-1}(\alpha_{j-1},\boldsymbol{x}_j,\boldsymbol{y}_j)
	\psi_j(\alpha_j,\boldsymbol{y}_j)\dd\boldsymbol{y}_j,
\]
we arrive at the desired singular value decomposition \eqref{eq:g_{j-1}}.

From
\[
  \partial_{\boldsymbol{y}}^{\boldsymbol\alpha}
  \partial_{\boldsymbol{y'}}^{\boldsymbol\beta}
  \kappa_j(\boldsymbol{y}_j,\boldsymbol{y}_j') 
  = \sum_{\alpha_{j-1}=1}^{r_{j-1}} \lambda_{j-1}(\alpha_{j-1})\int_{\Omega_j}
  \partial_{\boldsymbol{y}}^{\boldsymbol\alpha}\psi_{j-1}(\alpha_{j-1},\boldsymbol{x}_j,\boldsymbol{y}_j)
  \partial_{\boldsymbol{y'}}^{\boldsymbol\beta}\psi_{j-1}(\alpha_{j-1},\boldsymbol{x}_j,\boldsymbol{y}_j')\dd\boldsymbol{x}_j
\]
for any pair of multi-indices $\boldsymbol\alpha,\boldsymbol\beta\in\mathbb{N}_0^{m-j}$,
we conclude
\begin{align*}
\|\kappa_j\|_{{\bf H}^{\boldsymbol{t}_j}(\boldsymbol{\Upsilon}_j)
  	\otimes {\bf H}^{\boldsymbol{t}_j}(\boldsymbol{\Upsilon}_j)}
	&\le \sum_{\alpha_{j-1}=1}^{r_{j-1}} \lambda_{j-1}(\alpha_{j-1})
		\|\psi_{j-1}(\alpha_{j-1})\|_{L^2(\Omega_j)\otimes
			{\bf H}^{\boldsymbol{t}_j}(\boldsymbol\Upsilon_j)}^2\\
		%	=\|\boldsymbol{g}_j\|_{[L^2(\Omega_j)\otimes
		%	H_{mix}^{\boldsymbol{t}_j}(\boldsymbol\Upsilon_j)]^{r_{j-1}}}^2
%	&= \Bigg\|\sum_{\alpha_1=1}^{r_1}\sqrt{\lambda_1(\alpha_1)}
%		\varphi_1(\alpha_1)\otimes\psi_1(\alpha_1)\Bigg\|_{L^2(\Omega_1)\otimes 
%			H_{mix}^{\boldsymbol{k}}(\Omega_2\times\dots\times\Omega_m)}^2
	&=\|\boldsymbol{g}_{j-1}\|_{[L^2(\Omega_j)\otimes
			{\bf H}^{\boldsymbol{t}_j}(\boldsymbol\Upsilon_j)]^{r_{j-1}}}^2.
\end{align*}
Therefore, it holds $\kappa_j\in {\bf H}^{\boldsymbol{t}_j}(\boldsymbol{\Upsilon}_j)
\otimes {\bf H}^{\boldsymbol{t}_j}(\boldsymbol{\Upsilon}_j)$ which essentially 
implies the rate of convergence $r_j^{-2\beta_j}$ with
\begin{equation}\label{eq:betaj}
%================================
\beta_j = \min_{i=j+1}^m\bigg\{\frac{s_i}{n_i}\bigg\}
\end{equation}
in the spectral decomposition \eqref{eq:spectral}. This
in turn leads to the choice $r_j :=\varepsilon^{-1/\beta_j}$
to essentially derive the truncation error $\varepsilon$.

%===================================================
\subsection{Complexity}
%===================================================
We summarize our findings for the continuous tensor 
train decomposition in the following theorem, in which 
we estimate the ranks required.

\begin{theorem}\label{thm:ranks}
%=====================================
Let $f\in {\bf H}^{\boldsymbol{s}}(\boldsymbol\Omega)$ for 
some fixed $\boldsymbol{s}>{\bf 0}$ and $0<\varepsilon<1$. 
Choose the truncation ranks $r_j :=\varepsilon^{-1/\beta_j}$ 
for $j=1,\ldots,m-1$, where $\beta_j$ is given in \eqref{eq:beta1}
for $j=1$ and in \eqref{eq:betaj} for $j>1$. Then, the overall 
truncation error of the tensor train decomposition is essentially 
bounded by
\[
  \|f-f_{r_1,\ldots,r_{m-1}}^{\TT}\|_{L^2(\boldsymbol\Omega)}
  	\lesssim\sqrt{m}\varepsilon.
\]
The storage cost for $f_{r_1,\ldots,r_{m-1}}^{\TT}$
are given by
\begin{equation}\label{eq:costTT}
%==================================
  r_1 + \sum_{j=2}^{m-1} {r_{j-1} r_j}
  	= \varepsilon^{-1/\beta_1} + \varepsilon^{-1/\beta_1-1/\beta_2} 
		+\dots+ \varepsilon^{-1/\beta_{m-2}-1/\beta_{m-1}}
\end{equation}
and hence are bounded by $\mathcal{O}\big(m\varepsilon^{-2/
\min_{j=1}^{m-1}\{\beta_j\}}\big)$.
\end{theorem}

\begin{remark}\label{rem:ranks}
%=====================================
In case of $n:=n_1=\dots=n_m$ and $s:=s_1=\dots=s_m$, 
the cost of the tensor train decomposition is $\mathcal{O}
(m\varepsilon^{-2n/s})$. Thus, the cost is essentially independent
of the number $m$ of subdomains.
\end{remark}

%===================================================
\section{Concluding remarks}\label{sec:conclusion}
%===================================================
In this article, we have studied the convergence
of the truncated singular value decomposition for
functions from Sobolev spaces with dominating mixed
smoothness. With these convergence results at hand,
we proved that the particular ranks of the 
tensor train approximation are essentially independent 
of the overall dimension of the product 
domain. This is in contrast to the situation of approximating 
functions from the isotropic Sobolev spaces, where the 
maximum rank grows exponentially with the overall 
dimension, compare~\cite{GH21}.

We also like also to
comment on the relationship of the tensor train representation 
and deep neural networks (DNN). Indeed, it has been shown 
that one-dimensional basis functions, e.g.\ polynomials or 
wavelets, can be approximated up to an error $\varepsilon$
by DNNs with ReLU activation functions at cost $\sim |\log\varepsilon|$.
Moreover, the multiplication 
\[
(x,y)\mapsto x \cdot y = \frac{1}{4}\big\{(x+y)^2-(x-y)^2\big\}
\]
can easily be expressed by an additional layer and the 
$x\mapsto x^2$ activation function, which in turn can be 
approximated by a ReLU network with $\mathcal{O}(|\log\varepsilon|)$ 
layers, compare \cite{Yarotsky}. Vice versa,
the tensor train representation corresponds to a 
network with $m$ layers and the bivariate activation 
function $(x,y) \mapsto x \cdot y $, which itself can be 
represented by a multilayer ReLU network as described 
above. This implies that upper bounds for the DNN complexity 
can easily be derived from our presented results.

We finally want to 
highlight that the major progress of DNNs relies in the use of 
compositions of nonlinear functions as a tool for approximation.
In \cite{MP}, the authors introduced a class of ($m$-variate) 
functions which can be expressed by the composition of 
bivariate functions. Furthermore, it was shown that the overall 
complexity of a DNN for such functions is bounded by the number 
of compositions (i.e.\ layers) times the approximation complexity 
of the bivariate functions. The tensor train representation
is obviously contained in this class and the nonlinearity is 
simply the bilinear map. In the recent work \cite{BNS}, the 
authors have made the remarkable observation that 
the converse holds also true to a certain extent: The 
approximation rate of a corresponding tree based tensor 
network provides the same convergence rate as 
DNNs in the case of typical activation functions, 
as for example the ReLU activation function.

\bigskip
\noindent
{\bf Acknowledgment.} Michael Griebel was supported by 
the {\em Hausdorff Center for Mathematics} in Bonn, funded 
by the Deutsche Forschungsgemeinschaft (DFG, German 
Research Foundation) under Germany’s Excellence Strategy 
-EXC-2047/1-390685813 and the Sonderforschungsbereich 
1060 {\em The Mathematics of Emergent Effects}.

%=========================================================================

\end{document}